\documentclass{amsart}
\usepackage[utf8]{inputenc}
\usepackage{amsmath}
\usepackage{amsfonts}
\usepackage{amssymb}
\usepackage{amsthm}
\usepackage{graphicx}
\usepackage{epstopdf}
\usepackage{xcolor}
\usepackage{enumerate}

\title[Construction and optimization of simplicial meshes in $d$-dimensions]{Construction and shape optimization of simplicial meshes in $d$-dimensional space}
\author{Radim Ho\v sek}
\address{R.H., Department of Mathmatics, University of West Bohemia, 306 14 Plzen, Czech Republic}
\email{radhost@kma.zcu.cz}

\newcommand\ve{\mathbf{e}}

\newcommand\vp{\textbf{p}}
\newcommand\vw{\textbf{w}}
\newcommand\Rr{\mathbb{R}}
\newcommand\Nn{\mathbb{N}}
\newcommand\Zz{\mathbb{Z}}
\newcommand\diam{\mathrm{diam~}}
\newcommand\meas{\mathrm{meas}}
\newcommand{\ppi}{\boldsymbol\pi}
\newcommand{\dwtilde}{\widehat}

\newcommand{\rcolor}{\textcolor{blue}}

\newtheorem{thm}{Theorem}[section]
\newtheorem{lemma}[thm]{Lemma}
\newtheorem{prop}[thm]{Proposition}
\newtheorem{cor}[thm]{Corollary}
\newtheorem{note}[thm]{Remark}

\theoremstyle{definition}

\begin{document}
\maketitle

\numberwithin{equation}{section}
\begin{abstract}

We provide a constructive proof of a face-to-face simplicial partition of a $d$-dimensional space for arbitrary $d$ by generalizing the idea of Sommerville, used to create space-filling tetrahedra out of triangular base, to any dimension. Each step of construction that increases the dimension is determined up to a positive parameter, $d$-dimensional simplicial partition is therefore parametrized by $d$ parameters. We show the shape optimal value of those parameters and reveal that the shape optimal partition of $d$-dimensional space is constructed over the shape optimal partition of $(d-1)$-dimensional space.
\end{abstract}
\

{\bf Subj. AMS Class.:} 51M20, 51M04, 51M09, 65N50.

\section{Introduction}

\rcolor{
There has been introduced many tilings of a $d$-dimensional space, $d \in \Nn$, see for example a thorough summary of results on tilings by \emph{congruent} simplices in \cite{debrunner}. There has been also shown  that any unit $d$-dimensional cube can be decomposed into $d !$ simplices defined by
\begin{equation}\label{kuhn}
S_\pi = \{ x\in \Rr^d; 0 \leq x_{\pi(1)} \leq \dots \leq x_{\pi(d)} \leq 1 \}, \quad \pi \in \Pi_d,
\end{equation}
where $\Pi_d$ is the set of all permutations of numbers $1, \dots, d$. Moreover, these simplices have the same volume, $\meas_d S_\pi = (d!)^{-1}$.  See Kuhn's original paper \cite{kuhn} or the papers of Brandts et al. \cite{simplFE}, \cite{gradsupc}.  }

\rcolor{
But the not all partitions of the space need to use congruent simplices. When a simplicial partition of some general polyhedral domain satisfies the so called face-to-face property, it can be effectively used as a \emph{computational mesh} for various computational methods. A technique of such mesh generation can be found in \cite{redref}.}

A majority of today's computations take place in two or three spatial dimensions while those in higher dimension still occur rather rarely. However, some elliptic problems are treated in more dimension, see e.g. \cite{suli_elliptic} for such example emanating from stochastic analysis. Besides that, for problems represented by evolutionary partial differential equations of the hyperbolic type in three spatial dimensions, one can understand time as fourth variable and use a mesh in four-dimensional space, see e.g. the practical examples  \cite{volvo} and \cite{volvo2}. 

\rcolor{In this paper we introduce a method for creating a $d$-parametric family of tilings. Despite the set of parameters available, subsets of these tilings create only very rigid meshes.} However, some theoretical results suggest that for numerical methods to be convergent, the numerical domain and target domain do not necessarily have to coincide and that is where our meshes might find their use. Two different approaches can be found in works of Feireisl et al. \cite{FHMN}, \cite{FKM} and of Angot et al. \cite{angot}, \cite{khadra}. %Another advantage of our mesh are its low memory demands, as the vertices of the elements are distributed in a periodic pattern. 

Our result is strongly based on the (almost 100 years old) construction developed by Sommerville, which uses a regular triangle as a base for building a one-parametric family of tetrahedral elements that tile the three-dimensional space, see  \cite{goldberg}, \cite{wc} or the original Sommerville's article \cite{sommerville}. \rcolor{Our tilings are definitely not performed by congruent simplices and they do not cover $d$-dimensional cubes, thus they are clearly distinct from those introduced in \cite{debrunner} and \cite{kuhn}. }

We start the construction from one-dimensional simplices, i.e. segments, to increase the dimension repeatedly and  build a $d$-parametrical family of simplicial tessellations of $d$-dimensional space. Its existence is stated in Theorem \ref{thm:ex} and its proof covers Section \ref{sec:constr}. Then, in Section \ref{sec:optim} we determine the shape-optimizing vector of parameters with the result summarized in Theorem \ref{thm:optimal}. Section \ref{sec:concl} introduces some concluding remarks and open questions.

\section{Construction of the tessellation}\label{sec:constr}
We start with stating the existence result in the first of two central theorems of this article. 

\begin{thm}\label{thm:ex}
For any $d$-dimensional space there exists a $d$-parametric family of simplicial tessellations $\mathcal{T}_d(\mathbf{p}), \mathbf{p} = (p_1, p_2, \dots, p_d)$. For $\vp$ fixed, all elements $K \in \mathcal{T}_d(\vp)$ have the same $d$-dimensional measure equal to 

\begin{equation}
\meas_d K = \prod\limits_{i=1}^d \rcolor{|}p_i\rcolor{|}.
\end{equation}
 Moreover, every connected compact subset of the tessellation builds a face-to-face mesh. 
\end{thm}

\begin{figure}[h!]
\begin{center}
\includegraphics[scale=0.75]{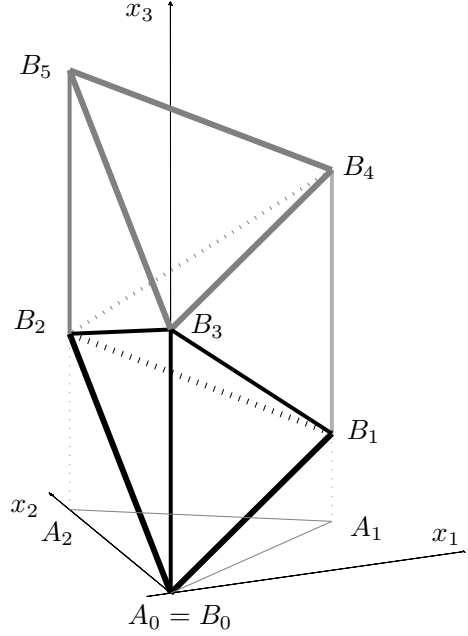}
\caption{Illustration of Sommerville's original construction creating a three-dimensional face-to-face mesh above unilateral triangular mesh. For the sake of clarity, only elements $K^{0,0}_2$ and $L^{0,0,0}_3, L^{0,0,1}_3, L^{0,0,2}_3$ are shown.}
\label{fig:23som}
\end{center}
\end{figure}

We start with introducing the original Sommerville's construction (see \cite{goldberg} or \cite{sommerville}) which creates a tessellation of an infinite triangular prism over an equilateral triangle (which tessellates the two dimensional space). In the construction, new vertices are created above (and below) the three vertices of the triangle in the heights $\dots,0, 3p, 6p, \dots$; $\dots, p, 4p, 7p, \dots$ and $\dots, 2p, 5p, \dots$, respectively, with a positive parameter $p$. Ordering these vertices with respect to their \emph{height} (i.e. third component), tetrahedra are defined as convex hulls of four consequent vertices. A sketch of this construction is given by Figure \ref{fig:23som}, with the notation given by upcoming Lemma \ref{lem:induction}, which is the key ingredient of Theorem \ref{thm:ex}. 

\begin{lemma}[Induction Step]\label{lem:induction}
Let $d\geq 2$ and $\mathcal{T}_{d-1} = \{ K^k_{d-1}\}_{k \in \Zz^{d-1}}$ be a simplicial tessellation of $(d-1)$-dimensional space such that the graph constructed from vertices and edges of $\mathcal{T}_{d-1}$ is a $d$-vertex-colorable graph. 

Then 
\begin{itemize}
\item there exists $\mathcal{T}_d = \{ L_d^l\}_{l \in \Zz^d}$ a simplicial tesselation of $d$-dimensional space with additional shape parameter $p_d$, 
\item any connected compact subset of $\mathcal{T}_d$ is a face-to-face mesh, 
\item $\mathcal{T}_d$ is a $(d+1)$-vertex-colorable graph.
\end{itemize}
\end{lemma}

%The construction of the proof is a generalization of the construction introduce by Sommerville for creating tetrahedral tessellation from regular triangulation of the two-dimensional space, see \cite{sommerville}, \cite{goldberg}, \cite{wc}. 

\begin{proof}
Take an element $K^k_{d-1} \in \mathcal{T}_{d-1}$, $K^k_{d-1} = \mathrm{co}\{A_0, A_1, \dots A_{d-1}\}$. Thanks to the ${d}$-vertex-colorability we can assume that the labels of vertices represent their color. Let $A_i = [A_{i,1}, A_{i,2}, \dots A_{i,d-1}]$ be the coordinates of $A_i$ in $(d-1)$-dimensional space. 

We define the following points in $d$-dimensional space:

\begin{equation*}
B_{j} = [A_{i(j),1}, A_{i(j),2}, \dots, A_{i(j),d-1}, jp_d], \quad j \in \Zz,  
\end{equation*}
where $i(j) \equiv j \mod d$ and $p_d > 0$ is a parameter. Denote 

\begin{equation}\label{new_el}
L^{k,z}_d = \mathrm{co}\{B_{z}, B_{z+1}, \dots, B_{z+d+1}\}, 
\end{equation}
the $d$-simplex as a convex hull of $d+1$ consequent vertices.
Then $\{L^{k,z}_d\}_{z \in \Zz}$ is a tessellation of an infinite $d$-dimensional prism with the cross-section $K^k_{d-1}$, see Figures \ref{fig:23som} and \ref{fig:12som} for illustration. As $\mathcal{T}_{d-1} = \{K^k_{d-1}\}_{k \in \Zz^{d-1}}$ is a tessellation of $(d-1)$-dimensional space, then the  set $\mathcal{T}_d := \{L^{k,z}_d\}_{(k,z) \in \Zz^{d-1}\times \Zz^d}$ forms a tessellation of $d$-dimensional space.

The construction uses the colors from the previous tessellation. Thus it is ensured that from any vertex $A_j$, that is shared by more simplices in $\mathcal{T}_{d-1}$, we create new vertices $V_{z}$ of only one type; having the last coordinate of the form 
\begin{equation}
V_{z,d} \frac{1}{p_d} \equiv c_{d-1}(A_j)\mod  d = j.
\end{equation}
This implies the face-to-face property, i.e. the facet of a simplex in tessellation $\mathcal{T}_d$ is a facet of another simplex. 

Finally, we define the new coloring with 

\begin{equation}\label{new_col}
c_d(B_j) \equiv j \mod d+1, \qquad \mbox{for~} B_j = [A_{i(j)},jp_d].
\end{equation}
Such mapping is a vertex coloring, since edges of the graph are only edges in simplices and vertices in any simplex $L^{k,z}_d$ have a different last component, but the `height' difference of two vertices connected by an edge does not exceed $dp_d$.
\end{proof}

\begin{figure}[h]
\begin{center}
\includegraphics[scale=0.75]{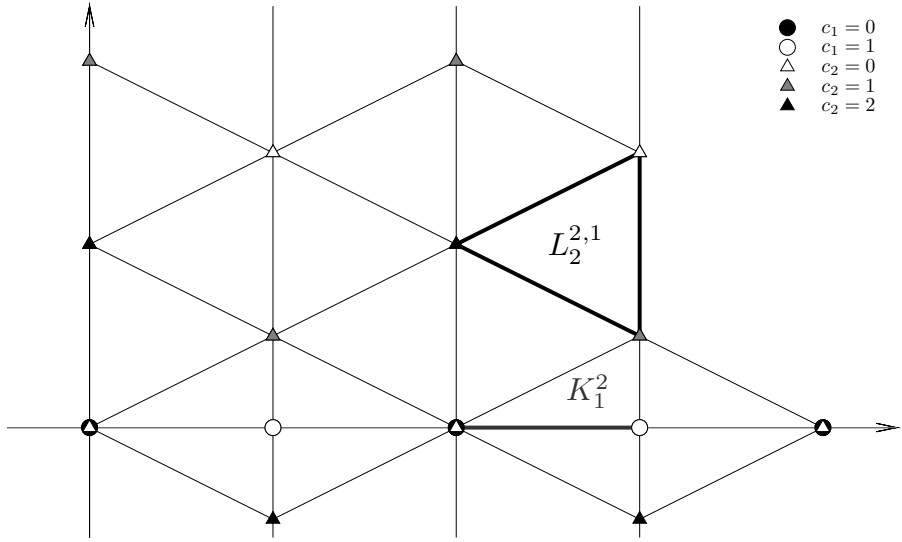}
\caption{Illustration of creating a simplicial face-to-face mesh of two dimensional space out of the one-dimensional one, with the parameters $p_1 = 1, p_2 = \frac{1}{2}$. The simplices $K^{2}_1$ and $L^{2,1}_2$ are marked in bold to clarify the notation defined by (\ref{new_el}). For general values of the parameters there are two candidates for diameter of $L^{k,z}_2$, equal to $\sqrt{p_1^2 + p_2^2}$ and $2p_2$. Notice also the vertex coloring, assigned through (\ref{new_col}). }
\label{fig:12som}
\end{center}
\end{figure}

The part that proves the face-to-face property based on vertex coloring of a graph was used already in \cite{wc}. Lemma \ref{lem:induction} supplies the induction step, to complete the proof of Theorem \ref{thm:ex}, we show the initial step.

\begin{proof}[Proof of Theorem 1]
A 1-dimensional Euclidean space (a line) can be divided into intervals of the length $p_1$. The color of a border point $A_z \in \{zp_1\}_{z \in \Zz}$ is given by

\begin{equation*}
c_1(A_z) \equiv z \mod  2.
\end{equation*}

The assumptions of Lemma \ref{lem:induction} are satisfied, hence we have the initial step and the induction step. \rcolor{For every use of Lemma \ref{lem:induction} we use the coloring that was generated by the Lemma in its previous use,} which finishes the proof. The equivolumetric property is proved by Proposition \ref{prop:volume}.  
\end{proof}

\rcolor{In the proof above, the considerate reader might be confused why we stressed that for the next step of construction the coloring produced by the previous use of the induction lemma is used. Clearly, at every step the original coloring $c_j$ can be changed using any $\pi_{j+1}$ a permutation  of numbers $\{0,\dots, j\}$. As a consequence, we may state the following.
\begin{thm}\label{thm:existence_perm}
For any $\vp = (p_1, \dots, p_d)$ and any vector $\ppi = (\pi_2, \dots, \pi_d)$, where $\pi_i \in \Pi_i$ is a permutation of numbers $\{0, \dots, i-1\}$, there exists a tessellation $\mathcal{T}_d(\vp, \ppi)$ of a $d$-dimensional Euclidean space. For $\vp$ fixed, all elements $K \in \mathcal{T}_d(\vp, \ppi)$ have the same $d$-dimensional measure equal to 
\begin{equation}
\meas_d K = \prod\limits_{i=1}^d |p_i|.
\end{equation}
 Moreover, every connected compact subset of the tessellation builds a face-to-face mesh. 
\end{thm}
Clearly, for a vector of identical permutations we get the original tessellation from Theorem \ref{thm:ex}, i.e. $\mathcal{T}_d(\vp, (Id, \dots, Id)) = \mathcal{T}_d(\vp)$. 
}

In general the created simplices are not identical. However, the following proposition shows that all elements of the tessellation $\mathcal{T}_d(\vp)$ have the same volume, i.e. the $d$-dimensional measure. 

\begin{prop}[Equal Volume of the Elements]\label{prop:volume}
Let $\mathcal{T}_d(\vp\rcolor{,\ppi})$ be the tessellation constructed by the procedure introduced in Proof of Lemma \ref{lem:induction}, with parameter vector $\mathbf{p} = (p_1,p_2, \dots, p_d)$ \rcolor{and vector of permutations $\ppi = (\pi_2, \dots, \pi_d)$}. Then for every simplex $L \in \mathcal{T}_d(\vp)$ we have

\begin{equation}\label{volume}
\meas_d L = \prod\limits_{i=1}^d \rcolor{|}p_i\rcolor{|}.
\end{equation}
\end{prop}

\begin{proof}
In one-dimensional space, the situation is obvious; points $zp_1, z\in \Zz$ divide a line into segments of the same length $\rcolor{|}p_1 \rcolor{|}$. We prove the induction step. Let us assume that there exists $M_{d-1}>0$ such that $\meas_{d-1} K = M_{d-1}$ for any $K \in \mathcal{T}_{d-1}$.

According to the construction, an element $L \in \mathcal{T}_d$ is determined by the points 

\begin{equation}\label{vol_ind_matrix}
\begin{split}
&B_z = [A_0, zp_d]; \quad B_{z+1} = [A_1, (z+1)p_d]; \quad \dots \\ & \qquad B_{z+d-1} = [A_{d-1}, (z+d-1)p_d]; \quad B_{z+d} = [A_0, (z+d)p_d],
\end{split}
\end{equation}
where $\mathrm{co}(A_0, A_1, \dots, A_{d-1})= K \in \mathcal{T}_{d-1}$.

The $d$-dimensional measure of a simplex is determined by the determinant of a matrix composed of the vectors that build the simplex, more precisely by the $(d!)^{-1}$ multiple of its absolute value. We use (\ref{vol_ind_matrix}) and performing operations that do not affect the value of the determinant we obtain

\begin{equation}\label{determinanty}
\begin{split}
\meas_d L = \frac{1}{d!} \left| \det  \begin{pmatrix}
  A_1-A_0 & p_d \\
  A_2-A_0 & 2p_d \\
  \vdots & \vdots \\
  A_{d-1}-A_0  & (d-1)p_d \\
  0 & dp_d
 \end{pmatrix} \right| = \frac{d\rcolor{|}p_d\rcolor{|} }{d!}  \left|  \det \begin{pmatrix}
  A_1-A_0   \\
  A_2-A_0   \\
  \vdots  \\
  A_{d-1}-A_0  
 \end{pmatrix}\right| &=  \\ = \rcolor{|}p_d\rcolor{|} \cdot \meas_{d-1} &K. \end{split}
\end{equation}

The proof is concluded by repeated use of (\ref{determinanty}) up to $d=1$, which yields (\ref{volume}).
\end{proof}

\begin{note}
\rcolor{In what follows, }we consider only positive values of $p_i$, shortly we write $\vp \in \Rr^d_+$, where $\Rr^d_+ = \left\{\mathbf{v}=(v_1, \dots, v_d) \in \Rr^d; v_i \geq 0, \forall i \in \{ 1,\dots, d\} \right\}$. It is rather a technical constraint, in fact one could allow $p_i \in \Rr\setminus\{0\}$. However, negative parameters affect only the orientation of the elements, not their shape characteristics. Therefore for the regularity optimization we can restrict ourselves to $\vp \in \Rr^d_+$ which also simplifies the process. \rcolor{One should bear in mind that if $\vp^\star = (p_1^\star, \dots, p_d^\star)$ is a vector of shape-optimal parameters, then also $(\delta_1 p_1^\star, \dots, \delta_d p_d^\star)$ is shape-optimal, for $\delta_j = \pm 1$. } 
\end{note}

\section{Regularity optimization}\label{sec:optim}
We have constructed a $d$-parametric family of tessellations in $d$-dimensional space, where the values of parameters $p_i, i=1, \dots, d$ influence their shape. We \rcolor{look for} a~vector of parameters $\vp^\star= (p^\star_{1}, \dots, p^\star_{d})$ for which the simplicial elements are \emph{shape optimal}. There are several regularity ratios with respect to which we might optimize. Some of them have been shown to be equivalent in the sense of the \emph{strong regularity} even in general dimension, see \cite{korotov}, but not in the sense of their maximization.

For convenient calculation we use the following ratio

\begin{equation}\label{ratio}
\vartheta(K) = \frac{\mathrm{meas}_d K}{(\mathrm{diam~} K)^d},\qquad d \geq 2, 
\end{equation}
where $\mathrm{meas}_d$ is the $d$-dimensional Lebesgue measure and $\diam K$ is the maximal distance of two points in $K$. The ratio $\vartheta(K)$ can be interpreted as a similarity of $K$ to an equilateral simplex. In other words, we find $\vp^\star$ and $K^\star$ which realize

\begin{equation}\label{optimization_bezpi}
\sup_{\vp \in \Rr^d_+} \min_{K \in \mathcal{T}_d(\vp)} \vartheta(K).
\end{equation}

As the simplices in $\mathcal{T}_d(\vp)$ are not identical, the optimization focuses on the \emph{worst} simplex only. Since we proved by Proposition \ref{prop:volume} that all elements in $\mathcal{T}_d(\vp)$ have the same $d$-measure, this \emph{worst case} in the sense of (\ref{ratio}) occurs when the diameter is maximal.

\subsection{\rcolor{Difficulties with the optimization}}
One can think through that the Sommerville's construction enables us to rewrite  (\ref{optimization}) using (\ref{volume}) as
\rcolor{
\begin{equation}\label{max_Wd_bez_pi}
\sup_{\vp \in \Rr^d_+} \min_{\vw \in \widetilde{W_d}} \frac{\prod_{i=1}^d p_i}{\left(\sum_{i=1}^d w_i^2 p_i^2\right)^{\frac{d}{2}}}, 
\end{equation}
where $\widetilde{W_d} \subseteq \dwtilde{W_d}$, which is defined by
\begin{equation}\label{dwWd}
\dwtilde{W_d} := \left\{\vw\in (\Nn \cup \{0\})^d \left| \exists k \in \{1,\dots,d\}: \begin{cases} w_k = k, \\ w_i = 0, \quad \mbox{for~} 1 \leq i < k, \\ w_j \in \{1, \dots,  j-1\}, \mbox{~for~} k < j \leq d  \end{cases} \right. \right\}.
\end{equation}
For example $\widetilde{W_3} = \dwtilde{W_3} =  \left\{(0,0,3), (0,2,1), (0,2,2), (1,1,2), (1,1,1) \right\}$ and $\widetilde{W_2} = \dwtilde{W_2} = \left\{(0,2),(1,1)\right\}$.
 However, in general $\widetilde{W_d} \not = \dwtilde{W_d}$ as the following Lemma shows.
\begin{lemma}\label{L:W4}
For $d=4$, the vector $(1,1,2,3) \in \dwtilde{W_4} \setminus \widetilde{W_4}$. In other words, there is no element $K \in \mathcal{T}_{4}(\vp)$ such that $\pm p_1 \ve_1 \pm p_2 \ve_2 \pm 2 p_3 \ve_3 \pm 3 p_4 \ve_4 $ in any combination of the signs is an edge of $K$. 
\end{lemma}}

\begin{proof}
\rcolor{Let $\vp' = (p_1, p_2, p_3)$ and $\vp = (\vp', p_4)$. Clearly there exists $L^{z_1, z_2, z_3}_3 \in \mathcal{T}_3(\vp')$ such that $\overrightarrow{UV} = \pm p_1 \ve_1 \pm p_2 \ve_2 \pm 2 p_3 \ve_3$ (in some combination of the signs) is an edge of $L^{z_1,z_2,z_3}_3$. (One of such elements is $L^{0,0,0}_3$, see Figure \ref{fig:23som}, for which $U = B_3, V = B_1$.) }

\rcolor{Then necessarily vertices $U,V$ have the height difference equal to $2 p_3$, i.e. $c_3(U) - c_3(V) \equiv  2 \mod 4$ and vertices created above $U$ and $V$ that belong to any $4$-simplex $K_4^{z_1, z_2, z_3, z_4} \in \mathcal{T}_4(\vp',p_4)$ have the difference vector equal to $\pm p_1 \ve_1 \pm p_2 \ve_2 \pm 2p_3 \ve_3 \pm  2p_4\ve_4$. As the construction in each step affects only the last component of the vector $\vw$, we conclude that $\vw \not \in \widetilde{W_4}$. }  
\end{proof}

\rcolor{The fact that $\widetilde{W_d} \not= \dwtilde{W_d}$ and problematic determination of their difference makes the optimization severely difficult. However, as the proof of the above lemma suggests, the difficulties are caused by inheriting the coloring from the preceding step of the construction. Removing this constraint by allowing recoloring before every step of the construction (as in Theorem \ref{thm:existence_perm}), we find out that 
\begin{equation}\label{optimization}
\sup_{\vp \in \Rr^d_+ } \min_{ \substack{ \ppi \in \Pi_2 \times \dots \times \Pi_d \\ K \in \mathcal{T}_d(\vp,\ppi)}} \vartheta(K)
\end{equation}
 is equivalent to 
 \begin{equation}\label{max_dwWd}
\sup_{\vp \in \Rr^d_+} \min_{\vw \in \dwtilde{W_d}} \frac{\prod_{i=1}^d p_i}{\left(\sum_{i=1}^d w_i^2 p_i^2\right)^{\frac{d}{2}}}, 
\end{equation}
 and also to
 \begin{equation}\label{max_Wd}
\sup_{\vp \in \Rr^d_+} \min_{\vw \in W_d} \frac{\prod_{i=1}^d p_i}{\left(\sum_{i=1}^d w_i^2 p_i^2\right)^{\frac{d}{2}}}, 
\end{equation}
where $W_d$ is defined by
\begin{equation}\label{Wd}
W_d := \left\{\vw\in (\Nn \cup \{0\})^d \left| \exists k \in \{1,\dots,d\}: \begin{cases} w_k = k, \\ w_i = 0, \quad \mbox{for~} 1 \leq i < k, \\ w_j = j-1, \mbox{~for~} k < j \leq d  \end{cases} \right. \right\}.
\end{equation}
}

\rcolor{The equivalence of the optimization problems (\ref{max_dwWd}) and (\ref{max_Wd}) is based on the facts that $(w_1, \dots, w_d)  \mapsto \sum_{i=1}^d w_i^2 p_i^2$ is increasing in each component and that elements in $W_d$ dominate those in $\dwtilde{W_d}$ componentwise.  } 

\rcolor{For example, for $d=3$ we have $W_3 = \{  \{1,1,2\},\{0,2,2\},\{0,0,3\}\}$.}

 Since $|W_d| =d$, we can also label its elements as $\vw_j = (w_{j,1}, w_{j,2}, \dots, w_{j,d}),$ where $j$ is its first nonzero coordinate. We also define

\begin{equation}\label{Dj}
D_j(\vp) = \sqrt{\sum_{i=1}^d w_{j,i}^2 p_i^2}  \qquad \mbox{and} \qquad D(\vp) = \max_{j \in \{1, \dots,d\}} D_j(\vp),
\end{equation}
so that (\ref{max_Wd}) can be rewritten as

\begin{equation}\label{max_D}
\sup_{\vp \in\Rr^d_+} \min_{k \in \{1,\dots, d\}} \frac{\prod_{i=1}^d p_i}{D_k(\vp)^d}.
\end{equation}

For illustration, we write out the `\rcolor{worst} diameter candidates' $D_j$ explicitly,

\begin{equation}\label{Dexample}
\begin{matrix}
& D_1(\vp)^2&=& p_1^2& +& p_2^2& +& 4p_3^2& + \dots& + (d-1)^2 p_d^2, \\
& D_2(\vp)^2 &=&  & & 4 p_2^2&+& 4p_3^2& + \dots& + (d-1)^2 p_d^2, \\
& D_3(\vp)^2 &=&   & & & & 9p_3^2 &+ \dots& + (d-1)^2 p_d^2, \\
& \vdots & & & & &\\
& D_{j-1}(\vp)^2 &=& (j-1)^2p_{j-1}^2 &+ & (j-1)^2p_j^2 &+& j^2 p_{j+1}^2 & + \dots& + (d-1)^2 p_d^2, \\
& D_j(\vp)^2 &=& & & j^2p_j^2 &+& j^2 p_{j+1}^2 & + \dots& + (d-1)^2 p_d^2, \\
&\vdots \\
& D_d(\vp)^2 &=& &   & &   & & &  d^2 p_d^2. \\
\end{matrix}
\end{equation}

\subsection{\rcolor{Optimal parameters}}
Now we can state the central theorem.

\begin{thm}[Optimal Parameters]\label{thm:optimal}
Let $d\geq 2$ and let $\mathcal{T}_d(\vp,\rcolor{\ppi})$ be a tessellation constructed through the procedure in \rcolor{proof of Theorem \ref{thm:existence_perm}}.  Then there exists a unique one-dimensional vector half-space
\begin{equation}\label{setP}
\begin{split}
P^\star = \Bigg\{&\vp^\star_\kappa \in \Rr^d_+| \vp^\star_\kappa = \kappa\vp^\star, \kappa > 0, \vp^\star = (p_1^\star, \dots, p_d^\star), \\ &p_1^\star = 1, p_2^\star = \frac{1}{\sqrt{3}}, p_j^\star = \frac{1}{j-1} \sqrt{\frac{2}{3}}, j \in \{3, \dots, d\}  \Bigg\},
\end{split}
\end{equation}
\vspace{0.3cm}of optimal parameters that realize

\begin{equation}\label{max_vartheta}
\sup_{\vp \in \Rr^d_+}  \min_{\substack{\rcolor{\ppi \in \Pi_2 \times \dots \times \Pi_d}   \\ K \in \mathcal{T}_d(\vp,\rcolor{\ppi})}} \frac{\meas_d K}{(\diam K)^d},
\end{equation}
\rcolor{ for some $\ppi \in \Pi_2 \times \dots \times \Pi_d$.} 
\end{thm}
\vspace{0.3cm} 
\rcolor{\begin{note}
Notice that we do not care much about ideal vector of permutations $\ppi$ nor the element $K$. The above result could also be interpreted as a lower bound on regularity ratio of elements $K(\vp^\star,Id) = K(\vp^\star)$ for (in some sense) \emph{shape optimal value} $\vp^\star$. 
\end{note}}
The rest of this section is devoted to the proof of Theorem \ref{thm:optimal}, which consists of three main steps. First, we prove the existence of the maximizer $\vp^\star$, then we show the particular form of the largest possible diameter that corresponds to the `most deformed' simplex in $\mathcal{T}_d(\vp^\star\rcolor{,\ppi^\star})$ and conclude the proof with determining the values of the components of $\vp^\star$ through constrained optimization. 

We would like to recall that we have three equivalent formulations of the optimization problem; (\ref{max_Wd}), (\ref{max_D}) and (\ref{max_vartheta}). 

\begin{lemma}[Existence of the Maximizer]\label{lem:max_ex}
Let $d\geq 2$. % and let $\mathcal{T}_d(\vp,\rcolor{\ppi})$ be a tessellation constructed through the procedure in Section \ref{sec:constr}. 
Then there exists a one-dimensional vector half-space
\begin{equation}
P^\star = \left\{\vp^\star_\kappa \in \Rr^d_+| \vp^\star_\kappa = \kappa\vp^\star, \kappa > 0 \right\},
\end{equation}
of optimal parameters that \rcolor{realize}

\begin{equation}\label{max_vartheta2}
\sup_{\vp \in \Rr^d_+} \rcolor{ \min_{\vw \in W_d}  \frac{\prod_{i=2}^d p_i}{\left(\sum_{i=1}^d w_i^2 p_i^2\right)^{\frac{d}{2}}} }. 
\end{equation}
\end{lemma}

\begin{proof}
As for the above discussion, (\ref{max_vartheta}) is equivalent to (\rcolor{\ref{max_vartheta2}}). We observe that the ratio in (\rcolor{\ref{max_vartheta2}}) is $0$-homogeneous, thus without loss of generality we  fix $p_1=1$. We continue with denoting the parametric vector by $\vp \in \Rr^{d}_+$, keeping in mind that due to its first component being fixed, $\vp$ may be considered as $(p_2,\dots, p_d) \in \Rr^{d-1}_+$. Defining 

\begin{equation*}%\label{Fp}
F(\vp):= \min_{\vw \in W_d} \frac{\prod_{i=2}^d p_i}{\left(\sum_{i=1}^d w_i^2 p_i^2\right)^{\frac{d}{2}}}, 
\end{equation*}
we can rewrite (\ref{max_vartheta2}) as $\sup_{\vp \in \Rr^d_+} F(\vp)$ and we observe that 

\begin{equation*}%\label{limits}
\lim_{p_j \to 0^+} F(\vp) = 0, \qquad \lim_{p_j \to \infty} F(\vp) = 0, 
\end{equation*}
for any $j \in \{2, \dots, d\}$. Moreover, $F \in C(\Rr^{d-1}_+)$ and $F>0$. Thus we infer that for any (sufficiently small) $\varepsilon$ the set $\Omega_\varepsilon:= \{F(\vp) \geq \varepsilon\}$ is a non-empty, bounded and closed subset of $\Rr^{d-1}_+$ and due to the continuity of $F$, it must attain its maximum in $\Omega_\varepsilon$ which necessarily coincides with the maximum of $F$ in $\Rr^{d-1}_+$. 
\end{proof}

In the next step we show which element of $W_d$ in (\ref{max_Wd}) or equivalently which $D_k$ in (\ref{max_D}) realizes the maximal diameter.

\begin{lemma}\label{lem:D1_active}
Let $\vp^\star = (1,p_2^\star,\dots, p_d^\star)$ be the maximizer of (\ref{max_D}). Then it holds that 
\begin{equation*}
D(\vp^\star) := \max_{k\in\{1,\dots, d\}} D_k(\vp^\star) = D_1(\vp^\star).
\end{equation*}
\end{lemma}

\begin{proof}
We proceed via contradiction. Let $D_1(\vp^\star)< D_k(\vp^\star) =D(\vp^\star)$ for some $k \in \{2, \dots, d\}$. Then we define $\vp' = (p_1',\dots, p_d')$ with

\begin{equation}\label{pstrich}
p_1' = 1, \qquad p_j' = p^\star_j \cdot \frac{1}{1+\delta}, \quad j \in \{2,\dots, d\},
\end{equation}
where $\delta>0$ is chosen small enough to ensure $D_1(\vp') < D_k(\vp')= D(\vp')$. Then it holds that 

\begin{equation}\label{Ddelta}
D(\vp') = D_k(\vp') = D_k(\vp^\star) \frac{1}{1+\delta} = D(\vp^\star) \frac{1}{1+\delta},
\end{equation}
\rcolor{as $w_j = 0$ for $j < k$, recall (\ref{Dj}), the definition of $D_k$.}
Substitution from (\ref{pstrich}) and (\ref{Ddelta}) into (\ref{max_Wd}) yields

\begin{equation*}
 \frac{\prod_{i=1}^d p_i'}{D_j(\vp')^d} = \frac{\prod_{i=1}^d p_i^\star}{D_j(\vp^\star)^d}\cdot\frac{(1+\delta)^d}{(1+\delta)^{d-1}} = (1+\delta)  \rcolor{\frac{\prod_{i=1}^d p_i^\star}{D_j(\vp^\star)^d}}, 
\end{equation*}
which contradicts the assumption of the maximality of $\vp^\star$.
\end{proof}

By virtue of Lemma \ref{lem:D1_active}, the maximization problem (\ref{max_D}), which is equivalent to (\ref{max_vartheta}), reduces to the optimization of a $C^1$ function with inequality constraints, 

\begin{equation}\label{max_D1}
\max \left\{ \left. \frac{\prod_{i=1}^d p_i}{D_1(\vp)^d} \, \right|  \vp \in \Rr^d_{+}, p_1 = 1, D_1(\vp)^2 \geq D_j(\vp)^2, \mbox{for~all~} j \in \{2, \dots, d\}  \right\} .
\end{equation}

\vspace{0.3cm}
To prove Theorem \ref{thm:optimal} it suffices to show that problem (\ref{max_D1}) has a unique solution, which is $\vp^\star$ in (\ref{setP}). By virtue of Lemma \ref{lem:D1_active} the optimization problem (\ref{max_D1}) is equivalent to (\ref{max_Wd}) and further to the original problem (\ref{max_vartheta}), hence Lemma \ref{lem:max_ex} guarantees it has a solution. 

The function 

\begin{equation}\label{F1}
F_1(\vp) = F_1(p_2,\dots, p_d) = \frac{\prod_{i=2}^d p_i}{D_1(1,p_2, \dots, p_d)^d},
\end{equation}
is continuously differentiable in $\Rr^{d-1}_+$, hence its constrained maximizer $\vp^\star$ satisfies the \emph{necessary} Karush-Kuhn-Tucker conditions. They read as follows,

\begin{equation}\label{kkt1}
\frac{\partial}{\partial p_j} F_1(\vp) =  \sum_{i=2 }^{d} \mu_i \frac{\partial}{\partial p_j}  \left( D_i(\vp)^2 - D_1(\vp)^2\right),
\end{equation}
\begin{equation}\label{kkt2}
\mu_j\left(D_j(\vp)^2 - D_1(\vp)^2\right) = 0,
\end{equation}
\begin{equation}\label{kkt3}
\mu_j \geq 0, \qquad D_j(\vp) \leq D_1(\vp),
\end{equation}
for $j =\{2, \dots, d\}$.

Let us focus on the right hand side of (\ref{kkt1}). Recalling (\ref{Dexample}) with $p_1= 1$, one can express

\begin{equation}\label{derDiD1}
\frac{\partial}{\partial p_j} (D_i(\vp)^2 - D_1(\vp)^2) = \begin{cases} 
-2(j-1)^2 p_j \quad & \mbox{for~} j < i, \\
2(2j-1)p_j & \mbox{for~} j=i, \\
0 & \mbox{for~} j > i.
\end{cases}
\end{equation} 

Then, by virtue of (\ref{F1}) and (\ref{Wd}) with (\ref{Dj}) and just derived (\ref{derDiD1}), we can rewrite (\ref{kkt1}) as

\begin{equation}\label{kkt}
\begin{split}
&\frac{\prod_{i=2}^d p_i}{D_1(\vp)^{2d}} \left( \frac{1}{p_j}  D_1(\vp)^d - d  (j-1)^2 D_1(\vp)^{d-2}p_j \right)   \\ & \quad  - 2 \mu_j (2j-1)p_j + 2(j-1)^2p_j \sum_{i=j+1}^d \mu_i = 0, \qquad  j \in \{2, \dots, d\}.
\end{split}
\end{equation}

It is not obvious how to get a solution of (\ref{kkt1}--\ref{kkt3}) or its equivalent (\ref{kkt2}, \ref{kkt3}, \ref{kkt}), nor its uniqueness. At the end, we show that $\mu_j = 0$ for $j \in \{3, \dots, d\}$ and $\mu_2 > 0$ which is then enough to determine uniquely the solution. To get this, we proceed in three steps. We show that

\begin{itemize}
\item there \emph{exists} $k \in \{2, \dots, d\}$ such that $\mu_k > 0$,
\item this $k$ is \emph{unique},
\item $k=2$. 
\end{itemize}

We introduce three lemmas, each corresponding to one of the items at the above list.

\begin{lemma}[Existence of an Active Constraint]\label{lem:list_ex}
Let $d \geq 2$ and $\vp^\star$ be the maximizer of (\ref{max_D1}). Then $\vp^\star$ is a solution of  (\ref{kkt1}--\ref{kkt3}) with $(\mu_2, \dots, \mu_d) \neq \mathbf{0}$, i.e. there exists $k \in \{2,\dots,d\}$ such that $\mu_k > 0$. 
\end{lemma}

\begin{proof}
We proceed via contradiction. Assume that $\mu_j = 0$ for all $j \in \{2, \dots, d\}$. In such case (\ref{kkt}), which is a consequence of (\ref{kkt1}), implies

\begin{equation*}%\label{pj_rovnost_2tod}
p_j^\star =  \frac{D_1(\vp^\star)}{(j-1)\sqrt{d}}, \quad j \in \{2, \dots, d\}, 
\end{equation*}
which substituted into $D_2(\vp)^2$ yields

\begin{equation*}
D_2(\vp^\star)^2 = \frac{4}{d} D_1(\vp^\star)^2 + \sum_{i=3}^d \frac{D_1(\vp^\star)^2}{d} = \frac{d+2}{d} D_1(\vp^\star)^2  > D_1(\vp^\star)^2,
\end{equation*}
which contradicts (\ref{kkt3}). Thus there is some $k \in \{2, \dots, d\}$ for which $\mu_k > 0$. 
\end{proof}

For $d=2$ Lemma \ref{lem:list_ex} implies directly that $k=2$. For $d \geq 3$ we supply the following lemma. 

\begin{lemma}[One Active Constraint]\label{lem:list_uniq}
Let $d \geq 3$ and $\vp^\star$ be a maximizer in (\ref{max_D1}) which satisfies (\ref{kkt1}--\ref{kkt3}) with $\mu_k > 0$ for some $k \in \{3, \dots, d\}$. Then $\mu_j = 0$ for all $j \in \{2, \dots, k-1, k+1, \dots, d\}$ and $\vp^\star = (1,p_2^\star, \dots, p_d^\star)$ fulfills

\begin{equation}\label{pstar_generalk}
p_j^\star = \begin{cases}
\displaystyle  \frac{\sqrt{2} D_k(\vp^\star)}{(j-1)\sqrt{dk}} \sqrt{\frac{2k-1}{k-1}}  \quad & \mbox{for~} j \in \{2, \dots, k-1\}, \\ ~\\
\displaystyle   \frac{D_k(\vp^\star)}{\sqrt{dk}} & \mbox{for~} j = k, \\ ~\\
\displaystyle   \frac{D_k(\vp^\star)}{(j-1)\sqrt{d}} & \mbox{for~} j \in \{k+1, \dots, d\}. 
\end{cases}
\end{equation} 
\end{lemma}

\begin{proof}
Let us take the largest $k \in \{3, \dots, d\}$ for which $\mu_k > 0$. Then for $j \in \{k+1, \dots, d\}$ we have $\mu_j = 0$. This enables us to deduce directly from (\ref{kkt}) that 

\begin{equation}\label{aux38}
p_j^\star = \frac{D_1(\vp^\star)}{(j-1)\sqrt{d}}, \quad j \in \{k+1, \dots, d\}.
\end{equation}

And as $D_1 = D_k$ (this follows from the assumption $\mu_k>0$ and (\ref{kkt2})) we can use (\ref{aux38}) for computing $p_k^\star$ \rcolor{from definition of $D_k$ (\ref{Dj})}. The computation

\begin{equation}\label{aux39}
D_k(\vp^\star)^2 = k^2 (p_k^\star)^2 + \sum_{j = k+1}^d (j-1)^2(p_j^\star)^2 = k^2(p_k^\star)^2 + \frac{d-k}{d} D_k(\vp^\star)^2,
\end{equation} 
yields 

\begin{equation}\label{aux40}
p_k^\star = \frac{D_k(\vp^\star)}{\sqrt{dk}}.
\end{equation}
Notice that (\ref{aux40}) holds even if $k = d$ and the summation in (\ref{aux39}) is void. 

Since $D(\vp^\star)=D_1(\vp^\star) = D_k(\vp^\star)$, then the constrained maximization problem (\ref{max_D1}) is equivalent to a constrained optimization, where $D_k$ is taken as the diameter, i.e.

\begin{equation}\label{max_Dk}
\max \left\{ \left. \frac{\prod_{i=1}^d p_i}{D_k(\vp)^d} \, \right|  \vp \in \Rr^d_{+}, p_1 = 1, D_k(\vp)^2 \geq D_j(\vp)^2, \mbox{for~all~} j \in \{1, \dots, d\}  \right\} .
\end{equation}

Arguing as before, the maximizer in (\ref{max_Dk}) exists and fulfills the following necessary Karush-Kuhn-Tucker conditions, 

\begin{equation}\label{kktk1}
\frac{\partial}{\partial p_j} \frac{\prod_{i=2}^d p_i}{D_k(\vp)^d} = \sum_{\substack{i=1 \\ i \neq k }}^d \nu_i \frac{\partial}{\partial p_j} (D_i(\vp)^2 - D_k(\vp)^2),
\end{equation}
for $j \in \{2, \dots, d\}$ and 
\begin{equation}\label{kktk2}
\nu_i(D_i(\vp)^2 - D_k(\vp)^2) = 0,
\end{equation}
\begin{equation}\label{kktk3}
\nu_i \geq 0, \qquad D_i(\vp) \leq D_k(\vp),
\end{equation}
for $i \in \{1,\dots,k-1, k+1, \dots, d\}$ and moreover we know that $D_1(\vp) = D_k(\vp)$. As we already settled $j \in \{k+1, \dots, d\}$, we need to focus on $j \in \{2, \dots, k-1\}$ only, hence we consider only those.

We know that 

\begin{equation}\label{aux41}
\frac{\partial}{\partial p_j} \frac{\prod_{i=2}^d p_i }{D_k(\vp)^d} = \frac{\prod_{i=2}^d p_i}{D_k(\vp)^d}\frac{1}{p_j}, \qquad j \in \{2, \dots, k-1\},
\end{equation}
and using (\ref{Dexample}) we compute the right-hand side of (\ref{kktk1}) for $j \in \{2, \dots, k-1\}$ as

\begin{equation}\label{aux42}
\frac{\partial}{\partial p_j} (D_i(\vp)^2 - D_k(\vp)^2) = \begin{cases}
2(j-1)^2p_j \quad &\mbox{for~} i < j, \\
2j^2p_j & \mbox{for~} i = j, \\
0 & \mbox{for~} i > j. 
\end{cases}
\end{equation}

Collecting (\ref{aux41}--\ref{aux42}) together with $\nu_i = 0$ for $i > k$ (as $D_i(\vp) < D_k(\vp)$ by assumption), we can rewrite (\ref{kktk1}) in the form

\begin{equation}\label{kktk}
\frac{\prod_{i=2}^d p_i}{D_k(\vp)^d} \frac{1}{p_j} = 2\nu_j j^2 p_j + 2 (j-1)^2 p_j \sum_{i=1}^{j-1} \nu_i,  \quad j \in \{2, \dots, k-1\}.
\end{equation}

Take any $j \in \{2, \dots, k-1\}$, we have either $\nu_j = 0$ or $\nu_j > 0$.

Let first assume $\nu_j = 0$. Then, from (\ref{kktk}) we deduce

\begin{equation}\label{pj_unc}
p_j^2 = p_{j,u}^2 = \frac{\prod_{i=2}^d p_i}{2 D_k(\vp)^d} \frac{1}{(j-1)^2 \sum_{i=1}^{j-1} \nu_i}. 
\end{equation}

If $\nu_j > 0$, then 

\begin{equation*}%\label{pj_c}
p_j^2 = p_{j,c}^2 = \frac{\prod_{i=2}^d p_i}{2 D_k(\vp)^d} \frac{1}{j^2\nu_j+ (j-1)^2 \sum_{i=1}^{j-1} \nu_i}. 
\end{equation*}

We observe that $p_{j,c} < p_{j,u}$ and $\vp^\star$ is supposed to maximize $\prod_{i=2}^d p_i \cdot (D_k(\vp))^{-d}$, where $D_k(\vp)$ is independent of $p_j$ for $j \in \{1,\dots, k-1\}$. Thus $p_j^\star$ needs to maximize only $\prod_{i=2}^d p_i$, i.e. only its value. \rcolor{Since $p_{j,u} > p_{j,c}$,} we choose its \emph{unconstrained} version $p^\star_j = p_{j,u}$ from (\ref{pj_unc}), i.e. $\nu_j = 0$ for any $j \in \{2, \dots, k-1\}$. This enables to rewrite (\ref{pj_unc}) into
 \begin{equation}\label{pj_unc2}
p_j^2 = p_{j,u}^2 = \frac{\prod_{i=2}^d p_i}{2 \nu_1 D_k(\vp)^d} \frac{1}{(j-1)^2}. 
\end{equation}

Computing (\ref{kktk1}) also for $j = k$, one gets

\begin{equation}\label{aux13}
\frac{1}{D_k(\vp)^{2d}} \left( \prod_{i=2}^d p_i D_k(\vp)^d \frac{1}{p_k}-  d \prod_{i=2}^d p_i D_k(\vp)^{d-2} k^2 p_k \right) = \nu_1 2 (-2k+1) p_k,  
\end{equation}
and after substituting $p_k^\star$ from (\ref{aux40}) into (\ref{aux13}) we can express $\nu_1$ as

\begin{equation}\label{nu1}
\nu_1 = \frac{dk \prod_{i=2}^d p_i}{2 D_k(\vp)^{d+2}} \frac{k-1}{2k-1}.
\end{equation}

Collecting (\ref{aux38}), (\ref{aux40}) and substituting from (\ref{nu1}) into (\ref{pj_unc2}) we get (\ref{pstar_generalk}), which concludes the proof.
\end{proof}

Lemmas \ref{lem:list_ex} and \ref{lem:list_uniq} give rise to the following corollary.

\begin{cor}\label{cor}
Let $d \geq \rcolor{2}$ and $\vp^\star$ be a maximizer in (\ref{max_D1}). Then there exists a~unique $k \in \{2, \dots, d\}$ such that $D_k(\vp^\star) = D_1(\vp^\star) = D(\vp^\star)$ and (\ref{pstar_generalk}) holds. 
\end{cor}

\begin{proof}
Lemma \ref{lem:list_ex} together with (\ref{kkt2}) gives existence of $k \in \{2, \dots, d\}$ such that $D_k(\vp^\star) = D_1(\vp^\star) = D(\vp^\star)$. \rcolor{For $d=2$ we get directly $k=2$.} For $k \geq 3$, Lemma \ref{lem:list_uniq} gives uniqueness of such $k$ and also (\ref{pstar_generalk}).  Using the procedure from the beginning of the proof of Lemma \ref{lem:list_uniq}, one recovers (\ref{pstar_generalk}) also for $\rcolor{d}=2$.
\end{proof}

Finally, we show that $k$ from the previous lemma is equal to $2$ which will enable us to determine also the values of $p_i^\star$.  

\begin{lemma}\label{lem:list_kis2}
Let $d \geq 2$ and $\vp^\star$ be a maximizer in (\ref{max_D1}). Then it holds that 

\begin{equation}\label{d1d2}
D(\vp^\star) = D_1(\vp^\star) = D_2(\vp^\star),
\end{equation}
and
\begin{equation}\label{pstar_final}
p_2^\star = \sqrt{\frac{1}{3}}, \qquad p_j^\star =\sqrt{\frac{2}{3}} \frac{1}{j-1},  \quad j \in \{3, \dots, d\}.
\end{equation}
\end{lemma}

\begin{proof}
Let $d=2$. Then Lemma \ref{lem:list_ex} implies (\ref{d1d2}), which can be written explicitly as $1 +(p_2^\star)^2 = 4(p_2^\star)^2$. Thus we infer $p_2^\star = 3^{-1/2}$. 

Let further $d\geq 3$. Then from Corollary \ref{cor} we get a unique existence of some $k \in \{2,\dots, d\}$ for which $D(\vp^\star) = D_1(\vp^\star) = D_k(\vp^\star)$ and the relation (\ref{pstar_generalk}) for $\vp^\star$. 

We prove $k=2$ via contradiction. Let us assume that $k \geq 3$. Then, $D(\vp^\star) = D_1(\vp^\star) = D_k(\vp^\star) > D_2(\vp^\star)$. Writing out $D_2(\vp^\star)$ explicitly using (\ref{pstar_generalk}), we get

\begin{equation}\label{aux61}
D^2 > D_2^2  = \frac{D^2}{d} \left(   2\frac{4(2k-1)}{k(k-1)} +   2\frac{(k-2)(2k-1)}{k(k-1)} +  \frac{(k-1)^2}{k} + (d-k) \right),
\end{equation}
\rcolor{where we skipped the argument $\vp^\star$ for the sake of brevity.}
Direct computation simplifies inequality (\ref{aux61}) into

\begin{equation*}
\frac{2k^2 + 9k-5}{k(k-1)} < 0, 
\end{equation*}
which is not true for any $k \in \Nn$, a contradiction. Thus $k=2$, and from (\ref{pstar_generalk}) we get 

\begin{equation}\label{aux62}
p_2^\star = \frac{D(\vp^\star)}{\sqrt{2d}}, \qquad p_j^\star = \frac{D(\vp^\star)}{(j-1)\sqrt{d}}, \quad j \in \{3,\dots, d\},
\end{equation}
which we substitute into $D_1(\vp)^2$ to get

\begin{equation}\label{aux63}
D(\vp^\star)^2  = D_1(\vp^\star)^2 = 1 + \frac{D(\vp^\star)^2}{2d} + (d-2)\frac{D(\vp^\star)^2}{d}.
\end{equation}

From (\ref{aux63}) we deduce $D(\vp^\star)^2 = \frac{2}{3}d$ which, substituted into (\ref{aux62}) yields (\ref{pstar_final}).
\end{proof}

\rcolor{
\begin{proof}[Proof of Theorem \ref{thm:optimal}]
The optimization problem (\ref{max_vartheta}) can be equivalently rewritten to (\ref{max_Wd}) and also to (\ref{max_D}). Lemma \ref{lem:max_ex} yields existence of the half-space of maximizers to (\ref{max_Wd}). Then, factoring the problem by fixing $p_1 = 1$, Lemma \ref{lem:D1_active} reduces (\ref{max_D}) to a constraint optimization problem (\ref{max_D1}). This problem is shown to have exactly one active constraint (Lemmas \ref{lem:list_ex}, \ref{lem:list_uniq} and Corollary \ref{cor}). Further, the active constraint is identified and the maximizer of (\ref{max_D1}) is determined in Lemma~\ref{lem:list_kis2}. Equivalence of the optimization problems concludes the proof.
\end{proof}
}

\section{Concluding remarks}\label{sec:concl}
We conclude with five remarks on various topics. 
\subsection{Optimization at each step}
Notice that the optimal values of parameters (\ref{pstar_final}) are independent of the dimension $d$. This can be interpreted that the most regular partition of $d$-dimensional space is constructed above the most regular partition of $(d-1)$-dimensional space. As a consequence, the shape optimization we performed is equivalent to the shape optimization at every dimension, which gives a sequence of one-dimensional optimization problems that is technically much less demanding. 
\subsection{Integer sequence for OEIS}
One can easily see that for suitable $\kappa$ it is possible to express the squares of the components of $\vp^\star_\kappa$ from (\ref{setP}) as fraction with unit numerator and integer denominator. Largest such $\kappa$, yielding the smallest possible integers in those fractions, is $\kappa = 2^{-1/2}$. For this value, the denominators give the following values:  2, 6, 12, 27, 48, 75, 108, 147, 192, 243, 300,\dots , having the formula for $j$-th item  $a_j = 3(j-1)^2$ for $j \geq 3$. This sequence has been \rcolor{upon the suggestion of the author} indexed in Sloane's database of integer sequences \cite{oeis} \rcolor{as sequence A289443}. 
\subsection{Shape optimality of the partition}
It is not obvious whether there exists any better simplicial tiling that cannot be constructed by our method. However, in 2D there is no triangle with better ratio $\vartheta$ than the equilateral one. Similarly, in 3D, our method gives the standard Sommerville tetrahedron (see \cite[Figure~2]{bf}), which as for Naylor \cite{naylor} is the best one among space-filling tetrahedra when considering the regularity ratio $\vartheta$. 

\rcolor{Moreover, we have computed that the regularity ratio of the \emph{worst} element in $\mathcal{T}_d(\vp^\star)$ is greater or equal to that of $\mathcal{T}_d(\vp^\star, \ppi^\star)$, which is, see (\ref{setP}), 
\begin{equation}
\frac{\prod_{i=1}^d p_i^\star}{D_1(\vp^\star)^d} =  \frac{\frac{\sqrt{3}}{2}\left(\frac{2}{3}\right)^d \frac{1}{(d-1)!} }{\left(\frac{2d}{3}\right)^\frac{d}{2}} = \frac{\sqrt{3}}{2} d \cdot \frac{1}{d^\frac{d}{2} d!},
\end{equation}
while for Kuhn's partition (\ref{kuhn}) we have 
\begin{equation}
 \vartheta(S_\pi) = \left(d^\frac{d}{2} d!\right)^{-1}.
 \end{equation} 
From this we can conclude that the elements of $\mathcal{T}_d(\vp^\star)$ are at least $\frac{\sqrt{3}}{2} d$ times more regular than those of Kuhn.}
%\subsection{The four-dimensional case}
%K\v r\' \i\v zek in \cite{5d} states that the question of the existence of a partition of 4-dimensional space  into acute simplices is open while Brandts et al. \cite[Conjecture 1.]{brandts-acute} expect that there will be no such partition. Using the partition 

%\begin{equation}
%\mathcal{T}_4(\vp^\star) = \mathcal{T}_4\left(1,\frac{1}{\sqrt{3}}, \frac{1}{\sqrt{6}}, \sqrt{\frac{2}{27}}\right), 
%\end{equation}
%one can verify, that the partition consists of simplices of a single type (and its reflections), whose largest dihedral angle equals $\pi/2$. Therefore we have a \emph{non-obtuse} partition of $\Rr^4$, leaving the question of existence of an acute one open. 
\subsection{Non-euclidean geometries}
We devote the last remark to the fact that the construction is independent of the underlying geometry and thus might be used also for computations in non-euclidean spaces. \rcolor{However, we cannot apply the optimization result directly, as it uses the equivolumetricity property. This is based on translation invariance which does not hold in non-euclidean geometries. }
More on tessellations of hyperbolic spaces can be found in works of Coxeter \cite{cox_book} or \cite{coxeter}, and Margenstern \cite{margenstern2003}, \cite{margenstern2006}, \cite{margenstern2011}. As Margenstern points out, these works might find their use in computational problems of theory of relativity or cosmological research, but such results had not been published before 2003 and to the best author's knowledge not even since these days.

\bibliographystyle{abbrv}
\bibliography{ddim_bib}

\end{document}